\documentclass[12pt]{article}
\usepackage{tikz}

\usepackage{amsmath}
\usepackage{amsfonts}
\usepackage{amssymb}

\newtheorem{definition}{Definition}
\newtheorem{remark}{Remark}

\newtheorem{fig}{Figure}

\newtheorem{theorem}{Theorem}

\newtheorem{proposition}{Proposition}

\newtheorem{conjecture}{Conjecture}
\newtheorem{problem}{Problem}

\newenvironment{proof}{\textbf{Proof} :}{$\square$}

\title{A Variation of Galvin and J\'onsson's approach to Sublattices of Free Lattices}
\author{Brian T. Chan}
\date{University of Calgary}

\begin{document}

\maketitle

\begin{abstract}
This article is part of my upcoming masters thesis which investigates the following open problem from the book, Free Lattices, by R.Freese, J.Jezek, and J.B. Nation published in 1995 \cite{FL}: ``Which lattices (and in particular which countable lattices) are sublattices of a free lattice?'' \\

Despite partial progress over the decades, the problem is still unsolved. There is emphasis on the countable case because the current body of knowledge on sublattices of free lattices is most concentrated on when these sublattices are countably infinite. \\

Galvin and J\'onsson's paper, \cite{DSFL}, characterized those sublattices of free lattices that are distributive by determining the distributive lattices which do not have any doubly reducible elements. In this article, a possible strategy towards solving this open problem is proposed. I will do this by extending a portion of Galvin and J\'onsson's paper \cite{DSFL}, about 54 years after its publication. To the best of my knowledge, the results derived in this article are new and not mentioned elsewhere in the literature. \\

Specifically, there appears to be a way to explicitly construct possible sublattices of free lattices in which: every proper subinterval is finite, the width of the lattice is finite. I will describe this approach I am developing in this article.
\end{abstract}

\section{Ladders, and related concepts}

Amongst lattices which are sublattices of free lattices the distributive laws are extremely strong. In 1959, F. Galvin and B. J\'onsson completely characterized those distributive lattices that are sublattices of free lattices in their paper \cite{DSFL}. Specifically, \\
 
\begin{theorem}[Galvin and J\'onsson, \cite{DSFL}] \label{the characterization}
 	A distributive lattice is (isomorphic to) a sublattice of a free lattice if and only if it is a countable lattice that is a countable linear sum of lattices where: each lattice in the linear sum is a one element lattice, an eight element boolean algebra $2 \times 2 \times 2$, or a direct product of the two element chain with a countable chain.
\end{theorem}
 
It should be noted that distributive lattices that are isomorphic to sublattices of free lattices must satisfy Whitman's condition. For that to happen, one has to require that no element in a distributive lattice is doubly reducible. Combining the distributive laws with this property turns out to be very restrictive. It played a crucial role in Galvin and J\'onsson's characterization: \\
 
\begin{theorem}[Galvin and J\'onsson, \cite{DSFL}] \label{the restriction}
 	A distributive lattice has no doubly reducible elements if and only if it is a linear sum of lattices where: each lattice in the linear sum is (isomorphic to) a one element lattice, an eight element boolean algebra $2 \times 2 \times 2$, or a direct product of the two element chain with a chain.
\end{theorem}
 
This property, that there are no doubly reducible elements, is implied by Whitman's condition. Moreover, the above theorem shows that any distributive lattice that has no doubly reducible elements automatically satisfies Whitman's condition. \emph{So their analysis in \cite{DSFL} can be interpreted as a study of distributive lattices that satisfy Whitman's condition}. \\

These three propositions, which are somewhat paraphrased from \cite{DSFL}, form the core of Galvin and J\'onsson's paper, the machinery on which their characterization is based on. \\

\begin{proposition}[Galvin and J\'onsson, \cite{DSFL} Lemma 5] \label{powerful}
	If a chain is a sublattice of a free lattice then that chain is countable.
\end{proposition}

\begin{proposition}[Galvin and J\'onsson, \cite{DSFL} Lemma 2] \label{two two two}
	A linearly indecomposable distributive lattice with no doubly reducible elements has a width of three if and only if it is isomorphic to $2 \times 2 \times 2$
\end{proposition}

\begin{proposition}[Galvin and J\'onsson, \cite{DSFL} Lemma 3] \label{ladder}
	A linearly indecomposable distributive lattice with no doubly reducible elements has a width of two if and only if it is isomorphic to $2 \times C$ where $C$ is a chain.
\end{proposition}

We refer the reader to \cite{DSFL} for proofs of these propositions. Proposition \ref{powerful} turns out to be very useful; it has been described in various lattice theory texts, including \cite{FL} and \cite{LTF}. I will make a new proof of proposition \ref{ladder} using the well-known $M_3$-$N_5$ theorem from lattice theory. Through that, we will outline a potential new approach to the open problem of this thesis. \\

In preparation for this new perspective, I will define two structures. This analysis appears to be quite effective for lattices of small (and finite) width. \\

\begin{definition}
	A ladder is a sublattice isomorphic to $2 \times \mathbb{Z}$, $\mathbb{Z}$ being the partial order: $\dots < -2 < -1 < 0 < 1 < 2 < \dots$
\end{definition}
\begin{definition}
	Let $L$ be a lattice and $a,b,c \in L$ be three distinct elements such that $b < c$, $a \parallel b$, and $a \parallel c$. Then a gadget, $G_L(a;b,c)$, is a sublattice of $L$ generated by $a$, $b$, and $c$.
\end{definition}

\begin{remark} \label{gadgets}
	Consider the free lattice on the partial order, $P$, with elements $\{A,B,C\}$ whose only comparability relation is $B < C$. Then the free lattice on $P$, $FL(P)$, is shown below. By the universal property of $FL(P)$, there is a unique lattice homomorphism $\phi: FL(P) \twoheadrightarrow G_L(a;b,c)$ where $\phi(A) = a$, $\phi(B) = b$, and $\phi(C) = c$. Hence, $G_L(a;b,c)$ is isomorphic to a quotient lattice of $FL(P)$. There are six different lattices that $G_L(a;b,c)$ can be isomorphic to representing a total of seven possible cases for $G_L(a;b,c)$. There are two cases that can occur when $G_L(a;b,c) \cong 2 \times 3$ and they are dual to each other.
	\begin{fig}
		\begin{tikzpicture}
		\node (A+C) at (0,3) {$A+C$};
		\node (A+B) at (-1,2) {$A+B$};
		\node (C) at (1,2) {$C$};
		\node (A+B C) at (0,1) {$(A+B)C$};
		\node (A) at (-2.5, 0.5) {$A$};
		\node (AC+B) at (0,0) {$AC+B$};
		\node (AC) at (-1,-1) {$AC$};
		\node (B) at (1,-1) {$B$};
		\node (AB) at (0,-2) {$AB$};
		
		\draw (A+B) -- (A+B C);
		\draw (A+C) -- (A+B) -- (A) -- (AC) -- (AC+B) -- (A+B C) -- (C) -- (A+C);
		\draw (AC) -- (AB) -- (B) -- (AC+B);
		\end{tikzpicture}
	\end{fig}
\end{remark}

\begin{remark} \label{gadgets in ladders}
	Adding elements to a ladder. Let $\{a,c\}$ be an antichain in a ladder, $H$, such that $a \prec a+c$ and $ac \prec c$ in $H$. Assume that $H$ is a sublattice of a lattice $L$, $b \in L \backslash H$, and $ac \prec b \prec c$ in $L$. Then we have $ac \leq ab \leq ac$. So with $ab = ac$ we see that the sublattice of $L$ generated by $H \cup \{b\}$ can be identified with one of three cases. 
	\begin{fig}
		\begin{tikzpicture}
		
		\node (a+b) at (0,0) {$a+c = a+b$};
		\node (c) at (1,-1) {$c$};
		\node (b) at (0,-2) {$b$};
		\node (a) at (-2,-2) {$a$};
		\node (ac) at (-1,-3) {$ac$};
		
		\draw (a+b) -- (a);
		\draw (ac) -- (b) -- (c);
		\draw (a) -- (ac);
		\draw (a+b) -- (c);
		
		\end{tikzpicture}
		\begin{tikzpicture}[scale=1.2]
		
		\node (a+c) at (0,0) {$a+c$};
		\node (a+b) at (-1,-1) {$a+b$};
		\node (c) at (1,-1) {$c$};
		\node (b) at (0,-2) {$b = (a+b)c$};
		\node (a) at (-2,-2) {$a$};
		\node (ac) at (-1,-3) {$ac$};
		
		\draw (a+c) -- (a+b) -- (a);
		\draw (ac) -- (b) -- (c);
		\draw (a) -- (ac);
		\draw (a+b) -- (b);
		\draw (a+c) -- (c);
		
		\end{tikzpicture}
		\begin{tikzpicture}
		
		\node (a+c) at (0,3) {$a+c$};
		\node (a+b) at (-1,2) {$a+b$};
		\node (c) at (1,2) {$c$};
		\node (a+b c) at (0,1) {$(a+b)c$};
		\node (a) at (-3, 0) {$a$};
		\node (b) at (-1,0) {$b$};
		\node (ac) at (-2,-1) {$ac$};
		
		\draw (a+b) -- (a+b c);
		\draw (a+c) -- (a+b) -- (a) -- (ac) -- (b) -- (a+b c) -- (c) -- (a+c);
		\end{tikzpicture}
	\end{fig}
	The third case can be thought of as being the first case with $a' = a$, $b' = b$, $c' = (a+b)c$, and $a'+c' = a+b$. Moreover, these three lattices are the three cases describing what $G_L(a;b,c)$ could be with $a$, $b$, and $c$ as specified in this remark.
\end{remark}

Here, I will make a new proof of an important lemma from Galvin and J\'onsson's paper (\cite{DSFL}, 1959), proposition \ref{ladder}. The new proof, in fact, generalizes to modular lattices and uses a well-known theorem in lattice theory known as the $M_3-N_5$ theorem. We want to both outline techniques that will assist in our investigation of infinite sublattices of free lattices and show the potential utility of such methods by demonstrating how it can be used towards the already known characterization of distributive sublattices of free lattices. We first introduce the $M_3 - N_5$ theorem. \\

The first part of the $M_3 - N_5$ theorem is due to R. Dedekind (1831 - 1916) and the second part is due to G. Birkhoff (1911 - 1996). A direct proof of this theorem is found in \cite{ILO}; as mentioned in \cite{ILO}, a more conceptual proof can be found within G. Gr\"atzer's texts on lattice theory. \\

The theorem involves a generalization of distributive lattices known as \emph{modular lattices}. A \emph{modular} lattice is a lattice where for all $a,b,c \in L$ with $a \leq c$, $a + bc = (a + b)c$. In an arbitrary lattice, this condition weakens to $a + bc \leq (a + b)c$ as $a + bc = ac + bc$ and $(a + b)c = (a + b)(a + c)$. For instance, the pentagon, $N_5$, is a lattice that is not modular. Since $N_5$ is also semidistributive, we see that modular lattices are not the same as semidistributive lattices. We now state the $M_3-N_5$ theorem: \\

\begin{theorem}[The $M_3$-$N_5$ Theorem, see \cite{ILO}]
	A lattice is modular iff it does not have a sublattice isomorphic to $N_5$ and distributive iff it does nor have a sublattice isomorphic to $N_5$ or to $M_3$.
\end{theorem}

A lattice $K$ is forbidden in $L$, that $K$ is a \emph{forbidden sublattice} of $L$, if $L$ does not have a sublattice isomorphic to $K$. The $M_3-N_5$ theorem is an example of a forbidden sublattice characterization. We note that the diamond, $M_3$, is modular but is not a semidistributive lattice. \emph{Whence, a lattice is distributive if and only if it is both modular and semidistributive. In particular, a sublattice of a free lattice is modular if and only if it is distributive.} We now move to proving proposition \ref{ladder}. \\

My alternative proof shown below uses a characterization that modular lattices (generalizations of distributive lattices) are precisely those lattices which do not have any sublattice isomorphic to $N_5$ ($N_5$ is forbidden); this being a consequence of the classical $M_3-N_5$ theorem (see \cite{ILO}). It should be noted that it makes explicit use of ZFC set theory, particularly the axiom of choice presented as the well ordering theorem. \\

Lemma $3$ of Galvin and J\'onsson's paper \cite{DSFL}, proposition \ref{ladder}, is the distributive case of the below assertion: \\

\begin{proposition}[Also see: \cite{DSFL} Lemma 3 from  Galvin and J\'onsson] \label{step 1}
	
	A linearly indecomposable modular lattice of width two that has no doubly reducible elements is isomorphic to the direct product of two chains, one of them being the chain with two elements.
	
\end{proposition}

\begin{proof}
	If the lattice is a sublattice of the direct product of two chains each having two elements then the above is trivially true. Conversely, assume that $L$ has more than four elements. The case when $|L| \leq 4$ can be seen to make $L \cong 2 \times 2$ \\
	
	Looking at the other possibilities, assume that there is a gadget $G_L(a;b,c)$ in $L$. To justify this step, suppose that $L$ does not have any gadgets. Then since $L$ is of width two, one can find $a \parallel b$ in $L$ where $x \in L \backslash \{a, b, ab, a+b \}$ implies that $x \geq a + b$ or $x \leq ab$. Then, by the linear indecomposability of $L$ with the facts that $ab$ is join irreducible and $a+b$ is meet irreducible, $L = \{a, b, ab, a+b \}$. \\
	
	Considering remark \ref{gadgets}, we note that since $L$ is modular, the $M_3-N_5$ theorem indicates that $N_5$ is a forbidden sublattice of $L$. Hence, $G_L(a;b,c)$ is the middle or right lattice below.
	
	\begin{fig} \label{three gadgets}
		
		\begin{tikzpicture}
		
		\node (u) at (0,0) {$u$};
		\node (t) at (-1,-1) {$t$};
		\node (z) at (1,-1) {$z$};
		\node (y) at (0,-2) {$y$};
		\node (s) at (-2,-2) {$s$};
		\node (x) at (-1,-3) {$x$};
		
		\draw (u) -- (t) -- (s);
		\draw (x) -- (y) -- (z);
		\draw (s) -- (x);
		\draw (t) -- (y);
		\draw (u) -- (z);
		
		\end{tikzpicture}
		\begin{tikzpicture}
		
		\node (a+c) at (0,0) {$a+c$};
		\node (a+b) at (-1,-1) {$a+b$};
		\node (c) at (1,-1) {$c$};
		\node (b) at (0,-2) {$b$};
		\node (a) at (-2,-2) {$a$};
		\node (ab) at (-1,-3) {$ab$};
		
		\draw (a+c) -- (a+b) -- (a);
		\draw (ab) -- (b) -- (c);
		\draw (a) -- (ab);
		\draw (a+b) -- (b);
		\draw (a+c) -- (c);
		
		\end{tikzpicture}
		\begin{tikzpicture}
		
		\node (a+c) at (0,0) {$a+c$};
		\node (a) at (-1,-1) {$a$};
		\node (c) at (1,-1) {$c$};
		\node (ac) at (0,-2) {$ac$};
		\node (b) at (2,-2) {$b$};
		\node (ab) at (1,-3) {$ab$};
		
		\draw (a+c) -- (c) -- (b);
		\draw (ab) -- (ac) -- (a);
		\draw (b) -- (ab);
		\draw (c) -- (ac);
		\draw (a+c) -- (a);
		
		\end{tikzpicture}
		
	\end{fig}
	
	Denote the elements of $G_L(a;b,c) \cong 2 \times 3$ using the left diagram in the above figure. We will show that $L \cong 2 \times C$ for some chain $C$ by constructing this direct product starting with $G_L(a;b,c)$. \\
	
	Firstly, let $w \in L \backslash G_L(a;b,c)$ and suppose that $z < w < u$. We see that $t \parallel w$ so consider the gadget $G_L(t;z,w)$. Since $L$ is modular, $G_L(t;z,w) \cong 2 \times 3$ and we have $y < tw < t$. But then $\{s, tw, z\}$ is an antichain contrary to assumption. Very similar arguments show that $x < w < s$ is impossible. And if $y < w < t$ then $\{s,w,z\}$ is an antichain, this being impossible. \\
	
	Because $L$ has no doubly reducible elements, $x$ is join irreducible and $u$ is meet irreducible. Hence if $w$ is an upper bound or a lower bound of $G_L(a;b,c)$ for all $w \in L \backslash G_L(a;b,c)$ then $L$ would not be linearly indecomposable, contrary to assumption. As $L$ has joins and meets, this implies that for all $w \in G_L(a;b,c)$, $w \nparallel w'$ for some $w' \in G_L(a;b,c)$. \\
	
	Without loss of generality, suppose that there is a $w \in L \backslash G_L(a;b,c)$ with $y < w$, $y \parallel t$ and $y \parallel z$. Then $\{w,t,z\}$ is an antichain in $L$, contrary to the assumption that the width of $L$ is two. \\
	
	Without loss of generality, suppose that there is a $w \in L \backslash G_L(a;b,c)$ with $w < z$, $w \parallel y$, and $w \parallel s$. Then $\{w,s,y\}$ is an antichain in $L$, contrary to assumption once again. \\
	
	Lastly, assume without loss of generality that there is a $w \in L \backslash G_L(a;b,c)$ such that $w < y$ and $w \parallel x$. Then as $s \parallel w$, consider the gadget $G_L(s;w,y)$. Because $w \parallel x$, this gadget is like the right diagram of figure \ref{three gadgets} with $a = s$, $b = w$, and $c = y$. Hence, as $x = sy$, $y = x + w$. But then $y = tz$ is doubly reducible, contrary to assumption. \\
	
	Hence, the sublattice generated by $G_L(a;b,c) \cup \{w\}$ is isomorphic to $2 \times 4$. By ZFC, we can use the well-ordering theorem to find a well-ordering of the elements $L\backslash G_L(a;b,c) = \{p_i : i \in \alpha\}$ for some ordinal $\alpha$. With this ordering, transfinite induction can be used. With $w$ above being $p_0$, apply the above argument when adding $p_i$ to $G_L(a;b,c) \cup \{p_j : j < i\}$; this works when $i$ is a successor or when $i$ is a limit ordinal.
\end{proof}

One advantage of this approach is that it leads to a generalization described later in this article. It starts with the below proposition, which is a variant of proposition \ref{step 1}. \\

With this proof technique I will make the following conjecture. It asserts that a certain class of lattices can be completely determined using an explicit construction. I suspect that the above arguments can be slightly modified to prove it, and that it is not particularly hard to prove.

\begin{conjecture} \label{conjecture 1}
	Let $L$ be a countable lattice of width two that satisfies Whitman's condition. Then $L$ is semidistributive and $L = \cup_{i \in N} P_i$. Here, $N$ is a chain that is order isomorphic to $\omega$, the dual of $\omega$, or $\dots < -2 < -1 < 0 < 1 < 2 < \dots$. And here, $P_i$ is the disjoint union of two chains $A$ and $B$ and $\{0_i,1_i\}$ where for all $a \in A$ and $b \in B$ $a + b = 1_i$ and $ab = 0_i$. Moreover, $P_i \cap P_j = \{0_j, 1_i\}$ if $i \prec j$ in $N$ and $P_i \cap P_j = \varnothing$ otherwise.
\end{conjecture}

J\'onsson's conjecture, proven by J.B. Nation in 1980, states that a finite lattice is a sublattice of a free lattice if and only if it is semidistributive and satisfies Whitman's condition. That conjecture is very deep. Below, I will make a similar conjecture which I suspect is a lot easier to prove. \\

\begin{conjecture} \label{conjecture 2}
	A width-two lattice is a sublattice of a free lattice if and only if it is semidistributive and satisfies Whitman's condition.
\end{conjecture}

\section{A possible classification}

In past research on sublattices of free lattices, a certain condition on a lattice introduced by B. J\'onsson had played an important role. The condition itself is explained in \cite{VL} and in \cite{FL}. With it, I will propose a possible way of looking at countable sublattices of free lattices at the end of this article. \\

We note that J\'onsson's condition ($L = D(L) = D^d(L)$) is a stronger form of the semidistributive laws. So now we describe the following definition from B. J\'onsson.

\begin{definition}[see \cite{VL} and \cite{FL}]
	Let $L$ be a lattice. Set $D_0(L)$ to be the set of join prime elements of $L$. Given $k \in \omega$, let $D_{k+1}(L)$ denote the set of elements $x$ with the property that every nontrivial join cover $X$ of $x$ refines to some join cover $X' \ll X$ of $x$ where $X' \subseteq D_k(L)$. Lastly, take $D(L) = \cup_{k \in \omega} D_k(L)$. Dually, define a sequence, starting with the meet prime elements and using dual join covers, $D_0^d(L), D_1^d(L), D_2^d(L), \dots$ with $D^d(L) = \cup_{k \in \omega} D_k^d(L)$.
\end{definition}

The following will be noted: \\

(1) The distributive sublattices, $L$, of free lattices which satisfy $L = D(L)$ and $L = D^d(L)$ are linear sums of chains, eight element boolean algebras, or a direct product $2 \times C$ \emph{where $C$ has both a least and a greatest element}. \\

(2) The distributive sublattices, $L$, of free lattices which satisfy $L = D(L)$ and $L \neq D^d(L)$ are as in (1) but also have, as a sublattice, a direct product $2 \times C$ \emph{where $C$ has a least and but not a greatest element}. \\

(3) The distributive sublattices, $L$, of free lattices which satisfy $L = D(L)$ and $L \neq D^d(L)$ are as in (1) but also have, as a sublattice, a direct product $2 \times C$ \emph{where $C$ has a greatest but not a least element}. \\

(4) The distributive sublattices, $L$, of free lattices which satisfy $L = D(L)$ and $L \neq D^d(L)$ are as in (1), (2), or (3) but also have, as a sublattice, a direct product $2 \times C$ \emph{where $C$ does not have a least element and does not have a greatest element}. \\

Using these observations I will make this classification, a partial order ordered by set inclusion. We will compare this with a similar diagram at the end of this article.

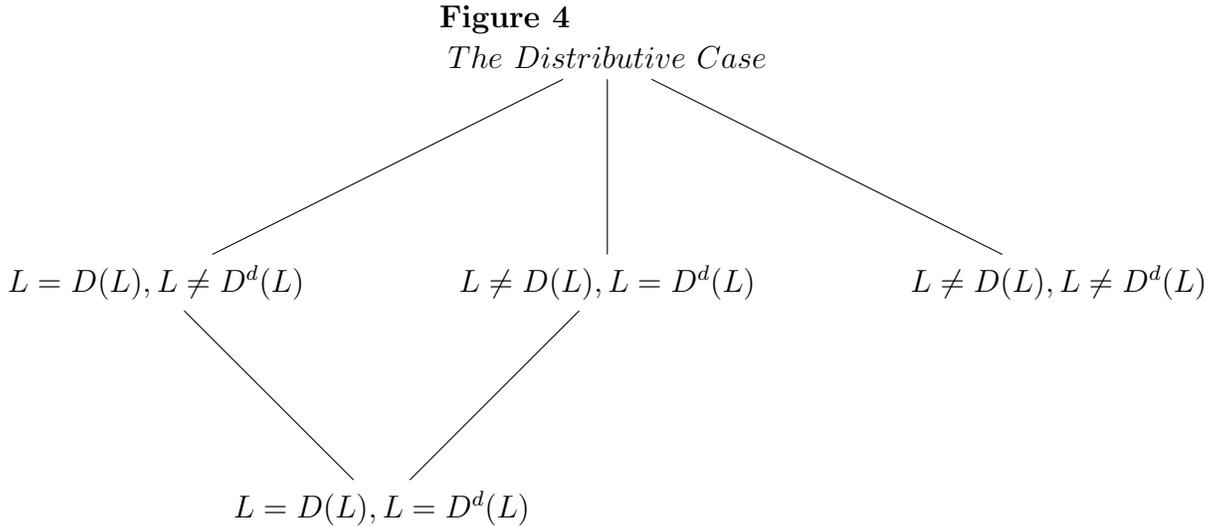
\begin{fig} \label{classification 1}
	\begin{center}
		\begin{tikzpicture}[scale = 3]
		\node (TDC) at (0,0) {$The$ $Distributive$ $Case$};
		\node (10) at (-2,-1) {$L = D(L), L \neq D^d(L)$};
		\node (01) at (0,-1) {$L \neq D(L), L = D^d(L)$};
		\node (11) at (-1,-2) {$L = D(L), L = D^d(L)$};
		\node (00) at (2,-1) {$L \neq D(L), L \neq D^d(L)$};
		
		\draw (TDC) -- (10) -- (11) -- (01) -- (TDC) -- (00);
		\end{tikzpicture}
	\end{center}
\end{fig}

\section{On semidistributive lattices satisfying Whitman's condition}

A very conservative extension of a part of Galvin and J\'onsson's 1959-1961 characterization of distributive sublattices of free lattices is derived. In the process, a nontrivial property of infinite semidistributive lattices satisfying Whitman's condition is obtained, assuming that every subinterval is finite. To the best of my knowledge, this small result I made is new and not noted elsewhere in the literature. With this, we complete our perspective on countable sublattices of free lattices using the semidistributive laws. \\

In this section, assume that $\mathbb{Z}$ has this order relation $\dots < -2 < -1 < 0 < 1 < 2 < \dots$ \\

The below results apply to the sublattices, $L$, of free lattices where every interval in $L$ is finite. Comparing with sharply transferable lattices, we note that every interval in a sharply transferable lattice, as described in the previous section, is finite. However, a sharply transferable lattice cannot have a spanning cover, which will be described later, so the result of this section does not directly apply to those lattices. This section is a (very) conservative extension of Galvin and J\'onsson's paper, \cite{DSFL}. It completes an outline of a possible future classification of countable sublattices of free lattices. \\

I will call a covering pair $a \prec b$ in $L$ a \emph{spanning cover} of $L$ if there are two sequences, $(a_i)$ and $(b_j)$, where $a < a_1 < a_2 < a_3 < \dots$ has no upper bound in $L$, $b > b_1 > b_2 > b_3 > \dots$ has no lower bound in $L$, $b \parallel a_k$ for all $k$, and $a \parallel b_k$ for all $k$. \\

Roughly speaking a sublattice, $L$, of a free lattice in which every interval is finite has the property that any spanning cover ``splits'' $L$ into two mutually disjoint pieces much like how $2 \times \mathbb{Z}$ ``splits'' into $\{(0,k) : k \in \mathbb{Z} \}$ and $\{(1,k) : k \in \mathbb{Z} \}$. This is possible because of the semi distributive laws. The word ``splitting'' is not the same as the term splitting used by R. Mackenzie in \cite{EBNMV}.

\begin{theorem} \label{spanning cover}
	Let $L$ be a countably infinite semidistributive lattice satisfying Whitman's condition. Assume that every interval in $L$ is finite. Then if $L$ has a spanning cover, $a \prec b$, there exists a partition, $L = A \cup B$, of $L$ with $a \in A$ and $b \in B$ with the following property. There is a sublattice, $H = \{(i,k) : i \in 2, j \in \mathbb{Z} \} \cong 2 \times \mathbb{Z}$ where $\{(0,k) \in H : k \in \mathbb{Z}\} \subseteq A$, $\{(1,k) \in H : k \in \mathbb{Z}\} \subseteq B$, $a = (0,0)$, and $b = (1,0)$. Moreover $B$ satisfies, and $A$ satisfies the dual of, the following:
	
	(1) For all $x \in B \backslash H$ and $k \in \mathbb{Z}$, $x \geq (0,k)$ implies $x \geq (1,k)$.
	
	(2) For all $x \in B \backslash H$, let $H_x$ be the sublattice of $L$ generated by $H \cup \{x\}$. Then $H_x \backslash H$ is finite.
	
	In particular, $A$ and $B$ are sublattices of $L$. One could call a partition $\{A,B\}$ of $L$ as described above a ladder splitting of $L$.
\end{theorem}

\begin{proof} The last part of this proof can be thought of as a proof sketch. Let $a \prec b$ be a spanning cover of $L$ with $L$ being as in the above theorem. Since $a \prec b$ is a spanning cover, let $a < a_1 < a_2 < a_3 < \dots$ have no upper bound in $L$, $b > b_1 > b_2 > b_3 > \dots$ have no lower bound in $L$, $b \parallel a_k$ for all $k$, and $a \parallel b_k$ for all $k$. Semidistributivity is used only in a few places, but it appears to remove many troublesome cases. So much so that a possible explicit construction of all lattices $L$ as specified above appears to be obtainable if $L$ is further assumed to have a width of three. \\
	
	We first construct $H$ as specified above. We note that $a_ib = a$ and $a + b_j = b$ for all $i$ and $j$. With this we find two subsequences $a < a_{i_1} < a_{i_2} < \dots$ and $b > b_{j_1} > b_{j_2} > \dots$ like so. Define $i_1$ to be the largest index $i$ such that $a_i + b = a + b$; this is possible since $L$ is assumed to have no infinite intervals. For positive integers $n$ define $i_{n+1}$ to be the largest index $i$ such that $a_i + b = a_{i_n+1} + b$. In a similar way, define $b > b_{j_1} > b_{j_2} > \dots$ from $(b_i)_i$. We abbreviate these sequences by setting $a'_n = a_{i_n}$ and $b'_n = b_{j_n}$. Now we define $(0,0) = a$ and $(1,0) = b$. For $n > 0$, define $(1,n) = a'_n + (1,0)$ and $(0,-n) = (0,0) b'_n$. \\
	
	To finish constructing our desired sublattice, we will need both $(SD_\vee)$ and $(SD_\wedge)$. Let $x_1$ be a coatom in $[a'_1, (1,1)]$. Since $(1,1)$ is the \emph{least} upper bound of $(0,1)$ and $a'_1$, $x_1 \ngeq (0,1)$. So we have $(0,1) > x_1 (0,1) \geq a'_1 (0,1) = (0,0)$. Hence, $x_1 (0,1) = (0,0)$. So we set $(0,1) = x_1$. Let $n$ be a positive integer and assume that $(0,n)$ has already been defined. Since $(0,n) (0,1) = (0,0)$ and $a'_{n+1} (0,1) = (0,0)$ we have $(a'_{n+1} + (0,n)) (0,1) = (0,0)$ by $(SD_\wedge)$. In particular, $a'_{n+1} \leq a'_{n+1} + (0,n) < (1,n+1)$. Now set $x_{n+1}$ to be a coatom in $[a'_{n+1} + (0,n), (1,n+1)]$ and repeat the above argument. In a similar way, $(SD_\vee)$ can be used to construct $\{(1,-n) : n = 1,2,3,\dots \}$. Since $(0,k)$ is covered by $(0,k+1)$ for all $k$, the resulting subset, $H = \{(i,k) : i \in 2, k \in \mathbb{Z} \}$, can be seen to be a sublattice of $L$ isomorphic to $2 \times \mathbb{Z}$. Moreover, $H$ has no upper or lower bounds in $L$. \\
	
	Here, only $(SD_\vee)$ or $(SD_\wedge)$ need be assumed. This application of the semidistributive laws appears to be very beneficial. Suppose that $x \in L \backslash H$ satisfies $(0,m) < x < (1,n)$ in $H'$ for some $m < n$ but that $x \nleq (1,m)$ and $x \ngeq (0,n)$. We note that because $[(0,m),(1m)]$ and $[(0,n),(1,n)]$ are prime intervals in $L$, $x + (0,n) = (1,n)$ and $(1,m)x = (0,m)$. But then, as $(1,m) + (0,n) = (1,n)$, $(SD_\vee)$ implies that $(0,n) = (1,m)x + (0,n) = (1,n)$. \\
	
	When $x \in L \backslash H$ and, for some $n$, that $(1,n) < x < (1,n+1)$, the above theorem can be seen to hold after considering the gadget $G_L((0,n+1);(1,n), x)$ where $(1,n) + (0,n+1) = x + (0,n+1) = (1,n+1)$ and remark \ref{gadgets in ladders}. The dual case automatically follows by symmetry. So lastly, assume without loss of generality that $x \in L \backslash H$ and, for some $m$, that $x > (1,m-1)$ yet $y \parallel (1,m)$. Considering the gadget $G_L((0,m);(1,m-1),x(1,m))$, where $(1,m-1) + (0,m) = x(1,m) + (0,m) = (1,m)$, remark \ref{gadgets in ladders} leads to the following generalized diagram for some $n > m$ \\
	
	\begin{fig}
		\begin{tikzpicture}[scale = 3]
		\node (1m-1) at (0,0) {$(1,m-1)$};
		\node (1m-1+x0m) at (0.5,0.25) {$(1,m-1) + x(0,m)$};
		
		\node (x1m) at (1.5,0.75) {$x(1,m)$};
		\node (x+1m) at (2.5,1.25) {$x + (1,m)$};
		\node (x) at (1.75, 1.5) {$x$};
		\node (1m) at (2.25,0.5) {$(1,m)$};
		
		\node (1n) at (3.5,1.75) {$(1,n)$};
		\node (0m-1) at (0.5,-1) {$(0,m-1)$};
		\node (x0m) at (1,-0.75) {$x(0,m)$};
		\node (0m) at (2.5,0) {$(0,m)$};
		\node (0n) at (4,0.75) {$(0,n)$};
		
		\draw (1m-1) -- (0m-1) -- (x0m) -- (0m) -- (0n) -- (1n) -- (x+1m) -- (x) -- (x1m) -- (1m-1+x0m) -- (1m-1);
		\draw (x1m) -- (1m) -- (x+1m);
		\draw (1m-1+x0m) -- (x0m);
		\draw (1m) -- (0m);
		\end{tikzpicture}
	\end{fig}
	
	This diagram is generalized, as the possibility when $x + (1,m) = (1,n)$ is also considered. This Hasse diagram is drawn at a slanted angle, the reader should tilt her/his head slightly to the right (in a later draft, the picture will, if necessary, be rotated). Firstly, assume that $(x + (1,m))(0,n) > (0,m)$. Then we consider $(1,m) + (x + (1,m))(0,n)$ and note that $x((1,m) + (x + (1,m))(0,n)) > x(1,m)$, for assuming otherwise would force $x((1,m) + (x + (1,m))(0,n)) < (1,m-1) + (0,m)$ which would violate Whitman's condition. \\
	
	We can then set $x' = x((1,m) + (x + (1,m))(0,n))$ and look at $x'$, $(1,m)$, $x' + (1,m)$, and $x'(1,m)$. From this we either obtain $(x'+(1,m)) > (x + (1,m))(0,n)$, $(x' + (1,m))(x + (1,m))(0,n) = (0,m)$, or $(x' + (1,m))(x + (1,m))(0,n) > (0,m)$. In the third case, we resort to the above argument to produce $x''$ and in the first two cases consider the below analysis. With $x^{(1)} = x'$, $x^{(2)} = x''$, $\dots$ we note the following. Since all intervals in $L$ are finite, only the first two cases will be possible for $x^{(N)}$ when $N$ is sufficiently large. \\
	
	Let $H' = (H \cup \{ (1,m-1) + x(0,m), x(0,m) \}) \backslash (\{ (0,k) : m < k < n \} \cup \{ (1,k) : m < k < n \})$. Certainly, $H' \cong 2 \times \mathbb{Z}$. To prove the theorem it is enough to show that when $(x + (1,m))(0,n) = (0,m)$ or $x + (1,m) > (0,n)$, $H' \cup \{x, x + (1,m), x(1,m) \}$ is a sublattice of $L$. \\
	
	To see this, we note that, for all $n' \geq n$, $(1,n') \geq x + (0,n') \geq (1,m-1) + (0,n') = (1,n')$ and that, for all $m' \geq m$, $x(0,m') = x(1,m)(0,m') = x(0,m)$.
	
\end{proof}

\begin{remark}
	An advantage of the above theorem, and the proof of this theorem, is that it provides a possible direction of attack when analysing infinite semidistributive lattices $L$, which satisfy $(W)$, have a spanning cover, \emph{and have finite width}. In particular, conjectures \ref{conjecture 1} and \ref{conjecture 2}, if correct, seem to, with the above, outline a method of determining all infinite semidistributive lattices $L$, which satisfy $(W)$, have a spanning cover, \emph{and have a width of three}. With such knowledge, it might be possible to determine countable sublattices of free lattices that have width three. A more ambitious endeavour may include finding counting sublattices of free lattices of finite width.
\end{remark}

A lattice does not necessarily have finite width. For example, consider the power set of an infinite set ordered by set inclusion or the free lattice $FL(3)$. Contrasting with these examples, is it possible that many examples of sublattices of free lattices that have spanning covers are ``not that much wider'' than a finite width lattice? We express this question more precisely below.

\begin{problem} \label{problem 1}
	Let $L$ be an infinite semidistributive lattice satisfying Whitman's condition where every interval in $L$ is finite and where $L$ has a spanning cover. Then is it true that every antichain in $L$ is finite?
\end{problem}

Consider the free lattice $FL(\kappa)$. This structure is semidistributive and satisfies Whitman's condition. When $\kappa \geq 3$ it is well-known that $FL(\kappa)$ has a sublattice isomorphic to $FL(\omega)$ and so has infinite antichains. If $\kappa$ is finite then $FL(\kappa)$ would have both a greatest and a least element; hence it is impossible for $FL(\kappa)$ to have a spanning cover. But if $\kappa$ is infinite, then any interval has either one element or infinitely many. To see this we use a standard argument from \cite{FL}, found in Corollary 9.11 of \cite{FL}. Let $v \leq u$ in $FL(\kappa)$, and for $x \in FL(\kappa)$ define var$(x)$ to be the set of elements that are in the generating set, $X$, of $FL(\kappa)$ and are present in any representation of $v$ as joins and meets of elements in $S$. Then $S = \{ (x \wedge u) \vee v : x \in X \backslash(\text{var}(u) \cup \text{var}(v))\} \subseteq [u,v]$ is infinite because both var$(u)$ and var $(v)$ are finite. Hence, it is impossible to find a covering pair, $a \prec b$, in $FL(\kappa)$. \\

We note that our investigation of countable sublattices, $L$, of a free lattice, where $L$ is assumed to have a spanning cover, originates from this class of lattices: Linearly indecomposble distributive sublattices, $L$, where $L$ is a sublattice of a free lattice and both $L \neq D(L)$ and $L \neq D^d(L)$. However, how many such countable sublattices actually satisfy $L \neq D(L)$ and $L \neq D^d(L)$? Replicating the above proof also proves that if $L$ has a width of $2$ and every interval of $L$ is finite then $L$ definitely satisfies $L \neq D(L)$ and $L \neq D^d(L)$. Unfortunately, such a class of lattices is very restricted as their diagrams can be obtained from $2 \times \mathbb{Z}$ by adding a finite number of points between pairs $(0,k), (0,k+1)$ or pairs $(1,k), (1,k+1)$. \\

Which lattices are countable sublattices of free lattices is not well-known. Any hope of even classifying such lattices has not been realized yet. For the open problem of this article I will sketch the following. It is an extension of figure \ref{classification 1}, since all distributive sublattice of free lattices are countable. Only the set inclusions known for this partial order are drawn. \\

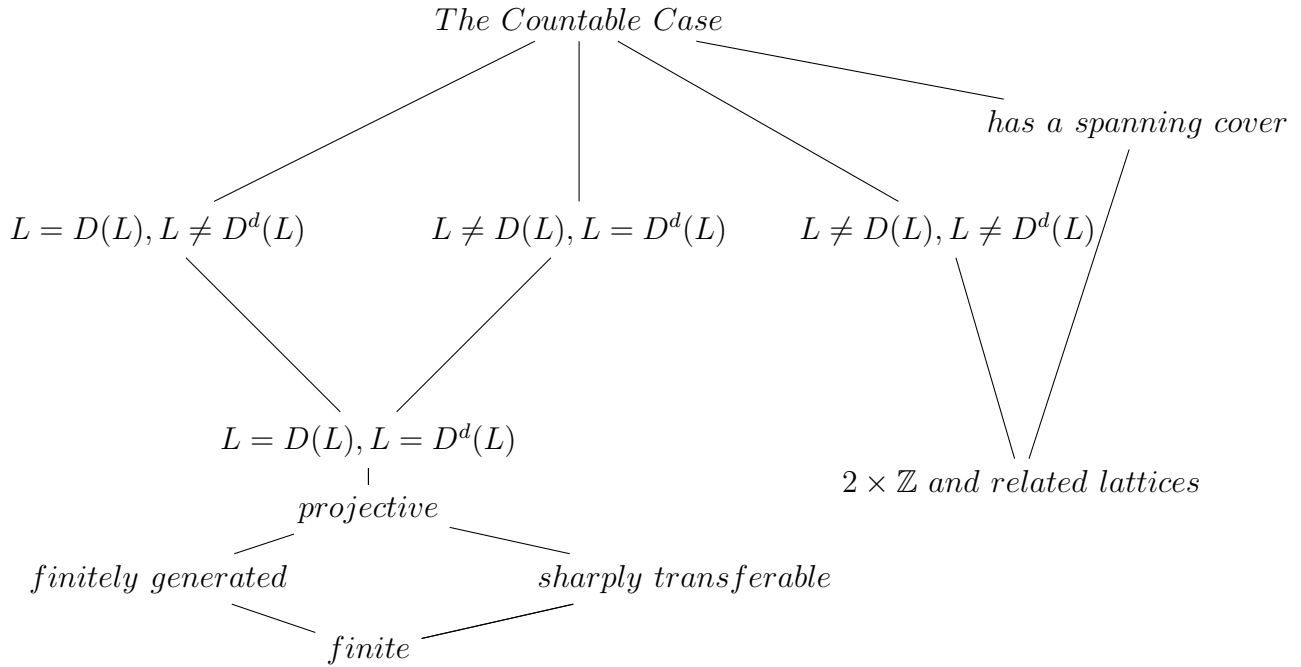
\begin{fig} \label{classification 2}
	\begin{center}
		\begin{tikzpicture}[scale = 2.8]
		\node (TCC) at (0,0) {$The$ $Countable$ $Case$};
		\node (10) at (-2,-1) {$L = D(L), L \neq D^d(L)$};
		\node (01) at (0,-1) {$L \neq D(L), L = D^d(L)$};
		\node (11) at (-1,-2) {$L = D(L), L = D^d(L)$};
		\node (00) at (1.75,-1) {$L \neq D(L), L \neq D^d(L)$};
		
		\draw (TCC) -- (10) -- (11) -- (01) -- (TCC) -- (00);
		
		\node (Proj) at (-1,-2.33) {$projective$};
		\node (ShTr) at (0.5,-2.66) {$sharply$ $transferable$};
		\node (FinGen) at (-2,-2.66) {$finitely$ $generated$};
		\node (Fin) at (-1,-3) {$finite$};
		
		\node (SpCov) at (2.65,-0.5) {$has$ $a$ $spanning$ $cover$};
		\node (SpCov2) at (2.1,-2.2) {$2 \times \mathbb{Z}$ $and$ $related$ $lattices$};
		
		\draw (11) -- (Proj) -- (FinGen) -- (Fin);
		\draw (11) -- (Proj) -- (ShTr) -- (Fin);
		\draw (00) -- (SpCov2);
		\draw (Fin) -- (ShTr);
		\draw (TCC) -- (SpCov) -- (SpCov2);
		\end{tikzpicture}
	\end{center}
\end{fig}

It is hoped that this diagram might assist in a possible future classification of such structures, which would create progress towards the countable case of the open problem. We conclude by quoting Reinhold at the end of his 1995 paper, \cite{WDLFL}: "\textit{\dots we are therefore still far away from a characterization of arbitrary sublattices of free lattices}"

\end{document}